\documentclass[a4paper,10pt]{amsart}
\usepackage[polish,english]{babel}
\usepackage[utf8]{inputenc}
\usepackage[T1]{fontenc}
\usepackage{polski}
\usepackage{multicol,breqn}
\usepackage{amsfonts, amsthm,wasysym}
\usepackage{enumerate}
\usepackage{aliascnt}  
\usepackage[pdftex,bookmarks,colorlinks]{hyperref}
\newtheorem{theorem}{Theorem}[section]

\newaliascnt{lemma}{theorem}  
\newtheorem{lemma}[lemma]{Lemma}  
\aliascntresetthe{lemma}

\newaliascnt{fact}{theorem}  
\newtheorem{fact}[fact]{Fact}  
\aliascntresetthe{fact}

\newaliascnt{proposition}{theorem}  
\newtheorem{proposition}[proposition]{Proposition}  
\aliascntresetthe{proposition}

\newaliascnt{corollary}{theorem}  
\newtheorem{corollary}[corollary]{Corollary}  
\aliascntresetthe{corollary}

\newtheorem{thm}{Theorem}[part]

\newaliascnt{appcor}{thm}  
\newtheorem{appcor}[appcor]{Corollary}  
\aliascntresetthe{appcor}

\theoremstyle{remark}\newtheorem*{rmk}{Remark}
\theoremstyle{definition}\newtheorem*{df}{Definition}

\newaliascnt{example}{part}
\theoremstyle{remark}\newtheorem{example}[example]{Example}
\aliascntresetthe{example}

\newcommand{\CSND}{\textbf{(CSND)}}
\newcommand{\CND}{\textbf{(CND)}}
\newcommand{\PD}{\textbf{(PD)}}
\newcommand{\SPD}{\textbf{(SPD)}}
\newcommand{\ls}{\langle}
\newcommand{\rs}{\rangle}

\newcommand{\dd}{\mathrm{d}}
\newcommand{\bigslant}[2]{{\raisebox{.25em}{$#1\!$}\left/\raisebox{-.25em}{$\!#2$}\right.}}

\DeclareMathOperator*{\fp}{\hexstar}

\author{Paweł Józiak}
\address{Institute of Mathematics of the Polish Academy of Sciences, Śniadeckich 8, 00-656 Warszawa, Poland}
\email{pjoziak@impan.pl}
\thanks{This paper was financed by the Polish National Science Center grant No. 2012/05/B/ST1/00626.}

\title{Conditionally strictly negative definite kernels}
\begin{document}
\selectlanguage{english}

\begin{abstract}
In this note we refine the notion of conditionally negative definite kernels to the notion of conditionally strictly negative definite kernels and study its properties. We show that the class of these kernels carries some surprising rigidity, in particular, the word metric function on Coxeter groups is conditionally strictly negative definite if and only if the group is a free product of a number of copies of $\bigslant{\mathbb{Z}}{2\mathbb{Z}}$'s and that the class of conditionally strictly negative definite kernels on a finite set is a one-parameter perturbation of the class of strictly positive definite kernels on this set. We also discuss several examples.
\end{abstract}
\maketitle

\section{Introduction}

The study of positive and negative (conditionally) definite kernels goes back to early results of A. Kolmogorov, I. Schoenberg and others. This simple notion has found several significant applications, for instance in proving the Haagerup approximation properties for several important classes of groups (the free groups, as shown by U. Haagerup in \cite{Ha}, Coxeter groups, as shown by M. Bożejko, T. Januszkiewicz and R. Spatzier in \cite{BJS} or groups acting on $\mathsf{CAT}(0)$ cube complexes, as shown by G. Niblo and L. Reeves in \cite{NR}), and it is typically a highly non-trivial result to show that a given kernel is of this type (for instance, the fact that the metric of the hyperbolic space $\mathbb{H}^n$ is conditionally negative definite required the effort done in \cite{FH} by J. Faraut and K. Harzallah). There are several equivalent definitions of a conditionally negative definite kernel (we will recall some of the definitions in \autoref{sec:2.1}). One of them is described by certain inequality. 
In this note 
we propose a slight strengthening of this notion, called conditional strict negative definiteness, requiring that the aforementioned inequality is strict, and study properties of such kernels (we state the precise formulation in the beginning of consecutive section). In particular, we describe these kernels in terms of quadratic embeddings into Hilbert spaces (as Schoenberg did for conditionally negative definite kernels in \cite{Sc1}). It turns out that the behavior of such kernels is very rigid, which is expressed e.g. in the following theorem.
\begin{theorem}\label{theorem:1}
 Let $X$ be a finite set and let $K\colon X\times X\to\mathbb{R}$ be such that $K(y,x)=K(x,y)\geq0$ and $K(x,x)=0$. If $K$ is conditionally stricly negative definite, then there exist a strictly positive definite kernel $A$ and a constant $c>0$ such that $K(x,y)=-A(x,y)+c$.
\end{theorem}
The above theorem shows that relaxing the assumption from being strictly positive definite to being \emph{conditionally} strictly negative definite in this case gives only a one-parameter perturbation of the \emph{a priori} much smaller family of kernels. Let us note that without the requirement of being conditionally \emph{stricly} negative definite this kind of theorem is not true (i.e. the class of positive definite kernels can be deformed in a much more complicated way, as it will be clear from examples included in \autoref{sec:2.1}).

Another point of view justifying the rigidity of the class is as follows. We consider the case of kernels defined by a function on a group (Toeplitz type kernels), the length functions are of greatest interest and appeared numerous times in the literature. The following hold:
\begin{theorem}\label{theorem:2}
 The word metric on a Coxeter group $\Gamma=\ls S|R\rs$ is conditionally strictly negative definite if and only if $\Gamma$ is a free Coxeter group (i.e. $\Gamma=\fp\limits_{i=1}^n\bigslant{\mathbb{Z}}{2\mathbb{Z}}$). The word metric on an Artin group $\Gamma=\ls S|R\rs$ is conditionally strictly negative definite if and only if $\Gamma$ is a free group (i.e. $\Gamma=\fp\limits_{i=1}^n\mathbb{Z}$). 
\end{theorem}

Note that while for general Coxeter groups the metric actually is conditionally negative definite (this was in fact the key result of \cite{BJS}), the same question for general Artin groups remains unsolved. Our result may thus be regarded as a specialization of the result from \cite{BJS} on one hand, and a statement about the length on general Artin groups on the other: even if their word metric was conditionally negative definite, it surely is not conditionally strictly negative definite, apart from the trivial example described above.

The question of examples of kernels satisfying the assumptions is investigated further in the paper, as well as the question about their permanence properties. We discuss mainly discrete metric spaces as examples (as in \autoref{theorem:2}). We show that the class of conditionally strictly negative definite kernels is closed under pointed sum, and hence closed under comb and star product of graphs. The key observation that leads to those results is based on a simple geometric description of these kernels (see \autoref{sec:2.3}).

The note is organized as follows: in \autoref{sec:2.1} we give all necessery definitions and recall some facts about kernels. In \autoref{sec:2.2} we briefly discuss some facts from linear algebra needful to clarify the notation. The key ingredients in our note are contained in \autoref{sec:2.3}, where we prove a lemma on quadratic embeddings of conditionally strictly negative definite kernels, and in \autoref{sec:2.4}, where we formulate and prove a useful construction lemma. Most of our examples rely on these two lemmas. The \autoref{sec:2.5} is devoted to the discussion of an obstruction to being a strictly negative definite kernel and we exploit it in \autoref{sec:3} to describe some graph-theoretic examples and non-examples, in particular, we give the proof of \autoref{theorem:2} and \autoref{theorem:1}. We conclude the note with an appendix not directly related to the arguments discussed in \autoref{sec:2}, giving a simple proof of Schoenberg's result showing that Euclidean metric is a conditionally 
strictly negative definite kernel and discuss a continuous analogue of \autoref{theorem:2}.

\vskip 1em\textbf{Acknowledgements.} Most of these results were included in the author's MSc thesis at University of Wrocław. We would like to thank Marek Bożejko for being a very patient advisor, inquiring interlocutor and benevolent friend; in particular: for directing the author towards the questions studied in this note. We would also like to thank Piotr Sołtan for his impact on the shape of this manuscript.

\section{Preliminaries}\label{sec:2}
\subsection{Definitions and generalities on kernels}\label{sec:2.1}
Let $c_c(X)$ denote the set of all complex-valued functions on a set $X$ with finite supports.
\begin{df}
A map $K\colon X\times X\to\mathbb{C}$ is called a \emph{kernel} on $X$. We say that a kernel $K$ is \emph{hermitian}, if $K(x,y)=\overline{K(y,x)}$. A hermitian kernel $K$ will be called \emph{Schoenberg kernel} if $K(x,y)=K(y,x)\geq0$ and $K(x,x)=0$. We say that a hermitian kernel $K$ is \emph{(strictly) positive definite}, if 
$$\forall\lambda\in c_c(X)\setminus\{0\}\sum\limits_{x,y\in V}\lambda(x)\overline{\lambda(y)}K(x,y)\geq0\quad(>0)$$
We say that $K$ is \emph{conditionally (strictly) negative definite}, if
$$\forall\lambda\in c_c(X)\setminus\{0\}\sum\limits_{x,y\in V} \lambda(x)\overline{\lambda(y)}K(x,y)\leq0\quad(<0)\quad\textrm{provided that }\sum\limits_{x\in X}\lambda(x)=0.$$
\end{df}
\noindent One can also introduce analogous notions of (strictly) negative definite kernels and conditionally (strictly) positive definite kernels, but they amount to a change of sign of $K$. We will follow the convention that the kernels that are conditionally definite are assumed to be conditionally negative definite, the kernels that are unconditionally definite are assumed to be positive definite. Notice that in literature the term ``conditionally'' is often omitted. Let us also denote the classes of such kernels by acronyms of their names in brackets, i.e. \CND, \CSND\, \PD, \SPD.

Let us notice that \SPD\ kernels appeared several times in the literature (\cite{CFS, CMS, Pin1, Pin2, Sun} and references therein), whereas \CSND\ kernels have not attracted enough attention so far, to our best knowledge. Let us also notice that, via Schur product and exponential function, pointwise exponential of a \CND\ kernel is a \PD\ kernel (\cite{Sc1}) and pointwise exponential of an \CSND\ kernel will be a \SPD\ kernel, altough the assumption is not necesary (i.e. one may find a \CND, non-\CSND, kernel whose pointwise exponential is \SPD, see \cite{Bo2}).
\begin{rmk}
Observe that for a hermitian kernel $K$ (which will be the assumption throughout the present article), it is enough to consider $\lambda\colon X\to\mathbb{R}$, since the imaginary part will always vanish.
\end{rmk}
\begin{example}\label{ex:0}
 Let $X$ be a set, let $d:X\times X\to\mathbb{R}$ be a pseudometric on $X$. Then $d$ is a Schoenberg kernel.
\end{example}
\begin{example}\label{ex:1}
Let $\mathcal{H}$ be a Hilbert space equipped with an inner product $\ls\cdot|\cdot\rs$, let $\alpha\colon X\to\mathcal{H}$ be any mapping. Then the kernel $K(x,y)=\ls\alpha(x)|\alpha(y)\rs$ is \PD. If moreover $(\alpha(x))_{x\in X}$ are linearly independent, the kernel $K$ is also \SPD.
\end{example}
\begin{rmk}
It is also a well-known result (attributed to A. Kolmogorov or being the essence of the famous Gelfand-Naimark-Segal construction) that for any abstract \PD\ kernel one can find a Hilbert space and a map $\alpha\colon X\to\mathcal{H}$ as in \autoref{ex:1}. 
\end{rmk}
\begin{example}\label{ex:2}
Let $A$ be any \PD\ kernel on $X$ and let $F\colon X\to\mathbb{C}$ be any function. Then $K(x,y)=-A(x,y)+F(x)+\overline{F(y)}$ is \CND. In fact, for any \CND\ kernel $K$ one can find a \PD\ kernel $A$ and a function $F$ so that the previous formula hold. 
\end{example}
\begin{example}\label{ex:3}
Following the setting from \autoref{ex:2}, if $A$ is \SPD, then $K$ is \CSND. The converse also holds in the following sense. Given a \CSND\ kernel $K$ defined on $X$, one can find a set $X'\supset X$ (in fact $X'\setminus X$ is just a single point), a kernel $A$ defined on $X'$ and a function $F\colon X'\to\mathbb{C}$ such that the restriction $A|_{X \times X}$ is a \SPD\ kernel and the equality above holds for all $x,y\in X$. The proof is just a slight modification of the argument showing the last assertion from \autoref{ex:2} 
\end{example}
\begin{rmk}
 One can regard \autoref{theorem:1} as a refinement of the converse construction from \autoref{ex:3}, for Schoenberg kernels on a finite set $X$: it turns out that the function $F$ can be chosen to be a constant function.
\end{rmk}
The \CND\ kernels we now call Schoenberg kernels were characterized by I. Schoenberg in the following theorem:
\begin{theorem}[\cite{Sc1}]\label{theorem:Sch1}
Let $K\colon X\times X\to\mathbb{R}$ be a Schoenberg kernel. If the kernel $K$ is \CND, then there exist a real Hilbert space $\mathcal{H}$ and a mapping $\alpha\colon X\to\mathcal{H}$ such that \begin{equation}\label{eq:*}\|\alpha(x)-\alpha(y)\|^2=K(x,y).\end{equation}
\end{theorem}

Important examples of kernels satisfying hypothesis of the theorem form a wide class of metric spaces (for which the metric is \CND), as in \autoref{ex:0}. The map $\alpha$ is often called the \emph{quadratic embedding} of a metric space into a Hilbert space (i.e. $d(x,y)=\|\alpha(x)-\alpha(y)\|^2$). We will call the map $\alpha$ a quadratic \emph{embedding} of a pair $(X,K)$, where $K$ is a \CND\ Schoenberg kernel, even if it is not a one-to-one map, whenever this does not lead to misunderstanding. It is clear that any mapping $\alpha\colon X\to\mathcal{H}$ induces a kernel $K\colon X\times X\to\mathbb{R}$ via formula \eqref{eq:*}, and that the kernel $K$ is \CND, so it is a characterization of \CND\ Schoenberg kernels. Observe that if a \CND\ kernel satisfies $K(x,x)\geq0$, one can form a kernel $K'(x,y)=K(x,y)-\delta_{x,y}K(x,x)$, which is a \CND\ Schoenberg kernel, thus the assumptions are not very restrictive. A natural question that we consider in following sections is whether the mapping $\alpha$ 
could have 
any additional properties provided that the kernel is \CSND. It turns out that this means exactly that the image of the quadratic embedding is an affinely independent subset of the Hilbert space. 

\subsection{Affine independence}\label{sec:2.2}
Let $V$ be a vector space over the real number field (generalization to an arbitrary field is straightforward) and fix a sequence of vectors $v_0,\ldots,v_n\in V$. We set
\begin{eqnarray*}
U & = & \{(\lambda_0,\ldots,\lambda_n)\in\mathbb{R}^{n+1}:\sum\limits_{k=0}^n \lambda_kv_k=0\},\\
W_n & = & \textrm{aff span}\{ v_0,\ldots,v_n\}=\biggl\{\sum\limits_{k=0}^n\lambda_kv_k:\sum\limits_{k=0}^n\lambda_k=1\biggr\},\\
\mathbb{R}^{n+1}_0 & = & \{(\lambda_0,\ldots,\lambda_n)\in\mathbb{R}^{n+1}:\sum\limits_{k=0}^n\lambda_k=0\}.
\end{eqnarray*}
It is clear that $U$ and $\mathbb{R}^{n+1}_0$ are vector spaces, while $W_n$ need not be, although there exist a vector $w\in W_n$ such that $W_n-w$ is a vector subspace of $V$ (unique if we also ask of $w$ to be of smallest norm among vectors satisfying this property). Define the affine dimension of $W_n$ as the linear dimension of $W_n-w$. We have the following
\begin{proposition}\label{lemma:affind}
 The following conditions are equivalent:
\begin{enumerate}[(a)]
 \item $\dim W_n = n$,
 \item The system of vectors $v_1-v_0,\ldots,v_n-v_0$ is linearly independent.
 \item $\mathbb{R}^{n+1}_0 \cap U = \{0\}$ (in particular, $U$ is of dimension at most one),
 \item there exists a unique $(n-1)$-dimensional sphere $S\subseteq V$ containing all the $v_j$'s.
\end{enumerate}\end{proposition}
Let us call a system of vectors $v_0,\ldots,v_n\in V$ satisfying the equivalent conditions of \autoref{lemma:affind} \emph{affinely independent}. While conditions \emph{(a)}, \emph{(b)}, \emph{(c)} are well known, the condition \emph{(d)} seems to be lesser known, so we give a proof for completeness.
\begin{proof}
\emph{(b)}$\iff$\emph{(d)}: It is clear that three points in a plane are either contained in a single line or contained in a single circle. We prove the assertion by induction on dimension: assume \emph{(b)}, i.e.  if $v_1-v_0,\ldots,v_n-v_0$ are linearly independent then so are $v_1-v_0,\ldots,v_{n-1}-v_0$ and by induction there is a unique $(n-2)$-dimensional sphere $S$ in $\{v_n-v_0\}^\bot$ (the orthogonal complement is regarded in standard scalar product) containing $v_1,\ldots,v_{n-1}$: call its center $x$. Consider a family of $(n-1)$-dimensional spheres $S_t$ with centers in $x+t\cdot(v_n-v_0)$ for $t\in\mathbb{R}$ and radii such that $\{v_n-v_0\}^\bot\cap S_t = S$ for all $t\in\mathbb{R}$ -- it is clear that for each $t$ there is only one such radius. Of course only for a single value of $t=\tau$ the sphere $S_\tau$ contains the point $v_n$ and if any other sphere $S'$ also satisfies this assertion, it satisfies also $\{v_n-v_0\}^\bot\cap S' = S$, so it coincides with $S_\tau$. Conversely: given $v_0,
\ldots,v_n$ not satisfying \emph{(b)} and given $(n-1)$-dimensional sphere $S$ satisfying condition \emph{(d)}, we can find in $V$ a line orthogonal to $\mathrm{span }\{ v_1-v_0,\ldots,v_n-v_0\}$ and moving the center of $S$ along this direction in the way described in previous part we get an infinite family of spheres containing $v_1,\ldots,v_n$, so \emph{d} is not satisfied.
\end{proof}
\begin{rmk}
 The assertion \emph{(d)} is purely finite dimensional: it is not true that for a given infinite set of vectors in a Hilbert space, for which all the finite subsets are affinely independent, there exists a sphere  containing  all of them. To see this it is enough to consider the set $\{k\cdot e_k:k\geq1\}$ with $(e_k)_k$ being the standard orthonormal basis of a separable Hilbert space. While for any $N\geq1$ there exist a unique $(N-2)$-dimensional  sphere in $\textrm{aff span}\{ke_k:1\leq k\leq N\}$ containing all the points $ke_k$ for $1\leq k\leq N$, it is easy to check that the radii of the spheres, as well as the norms of their centers, have asymptotic $\mathcal{O}(N^{3/2})$. 
 \end{rmk}
 \noindent An immediate consequence of \autoref{lemma:affind} is the following:
\begin{fact}\label{lemma:fin}
 Let $v_0,\ldots,v_n\in V$ be a sequence of affinely independent vectors (in particular, they are pairwise distinct), let $W=\textrm{aff span}\{ v_0,\ldots,v_n\}$ be their affine span. If $0\notin W$, then $v_1,\ldots,v_n$ are linearly independent.
\end{fact}

\subsection{A lemma on quadratic embedding}\label{sec:2.3}
\begin{lemma}\label{lemma:keylem}
Let $(X,K)$ be a finite set with a \CND\ Schoenberg kernel, let $\alpha\colon V\to\mathcal{H}$ be its quadratic embedding. The following conditions are equivalent:
\begin{enumerate}[(a)]
 \item $0\notin\sigma(\tilde{K})$, where $\tilde{K}: \ell^2(X)\to\ell^2(X)$ is defined as $\tilde{K}\delta_x=\sum\limits_{y\in X}K(x,y)\delta_y$ (in other words, the matrix defined by $K$ is invertible),
 \item The set of vectors $\{ \alpha(x): x\in X\}$ is affinely independent,
 \item $K$ is \CSND.
\end{enumerate}
\end{lemma}
\begin{proof}
\emph{(c)}$\iff$\emph{(a)}: this is a simple result in linear algebra or a corollary of Courant-Fischer-Weyl min-max principle , cf. \cite[Theorem~XIII.1]{SR}. \vskip 1em
\emph{(c)}$\iff$\emph{(b)}: since 
$K(x,y)=\|\alpha(x)-\alpha(y)\|^2 = \|\alpha(x)\|^2+\|\alpha(y)\|^2 -2\ls\alpha(x)\big|\alpha(y)\rs$, we have:
\begin{multline}\notag
 2\|\sum\limits_{x\in X} \lambda(x)\alpha(x)\|^2=2\sum\limits_{x,y\in X}\lambda(x)\lambda(y)\big\ls\alpha(x)\big|\alpha(y)\big\rs\\
 =\sum\limits_{x,y\in X}\lambda(x)\lambda(y)\bigg(\|\alpha(x)\|^2+\|\alpha(y)\|^2-K(x,y)\bigg)\\
 =\bigg(\sum\limits_{x\in X}\lambda(x)\bigg)\bigg(\sum\limits_{x\in X} \lambda(x)\|\alpha(x)\|^2\bigg) + \bigg(\sum\limits_{x\in X}\lambda(x)\bigg)\bigg(\sum\limits_{x\in X} \lambda(x)\|\alpha(x)\|^2\bigg) \\
 -\sum\limits_{x,y\in X}\lambda(x)\lambda(y)K(x,y),
\end{multline}
so for any $\lambda\colon X\to\mathbb{R}$ satisfying $\sum\limits_{x\in X}\lambda(x)=0$ we have \[\sum\limits_{x,y\in X}\lambda(x)\lambda(y)K(x,y)\leq0.\] Thus, by \autoref{lemma:affind}\emph{(c)}, $K$ is \CSND\ if and only if the vectors $\{\alpha(x):x\in X\}$ are affinely independent.
\end{proof}
\begin{corollary}
 The minimal dimension of the Hilbert space $\mathcal{H}$ admitting the quadratic embedding of a \CSND\ kernel on a set $X$ is equal to $\# X-1$.
\end{corollary}

\subsection{Construction lemma}\label{sec:2.4}
The construction discussed in the present section is a simple adaptation of a construction given in \cite{Bo1}. For $i=1,2$, let $X_i$ be a set, $e_i\in X_i$ a distinguished point and let $K_i$ be a kernel on $X_i$. Consider a pointed set $X= X_1\!{}_{\phantom{1}e_1}\!\star_{e_2} X_2:=X_1\sqcup X_2 \slash\!\!\!\sim$, where the relation is given by $e_1\sim e_2$. In the graph theory this construction is usually called \emph{the Markov sum of graphs}, while in the topology \emph{wedge sum of topological spaces}. We will denote this space simply by $X_1\!\star X_2$. Now, we define a kernel on $X$:
\begin{displaymath}
K(x,y) = \left\{ \begin{array}{lll}
K_1(x,y) & \textrm{if} & x,y\in X_1,\\
K_1(x,e_1)+K_2(e_2,y)& \textrm{if} & x\in X_1, y\in X_2,\\
K_2(x,e_2)+K_1(e_1,y)& \textrm{if} & x\in X_2, y\in X_1,\\
K_2(x,y) & \textrm{if} & x,y\in X_2,\\
\end{array} \right.  
 \end{displaymath}
 The Markov sum is important from the point of view of free probability and spectral analysis on graphs: the adjacency matrix of a graph $X$ which is a Markov sum of two other graphs, viewed as an operator on $\ell^2(X)$ can be represented as a sum of two boolean independent operators: the adjacency matrices of the summand graphs (in the natural vacuum state). It also applies to other graph-theoretic constructions like the comb product (associated with monotone independence) or the free product (associated with free independence), a nice reference for those is \cite{HO}. One has the following 
\begin{lemma}\label{lemma:ms}
 If two kernels $K_i$ defined on $X_i\ni e_i$, $i=1,2$, are \CSND\ Schoenberg kernels, so is $K$ on $X_1\!\star X_2$.
\end{lemma}
\begin{proof}
Let $\alpha_i\colon X_i\to\mathcal{H}_i$, $i=1,2$, be two quadratic embeddings associated to these kernels. Since composition with a translation preserves \eqref{eq:*}, one may assume that $\alpha_i(e_i)=0$. Define $\alpha=\alpha_1\oplus\alpha_2\colon X_1\star X_2\to\mathcal{H}_1\oplus\mathcal{H}_2$, it is clear that the kernel induced by the map $\alpha$ via the equality \eqref{eq:*} is precisely $K$: this is due to the definition of the kernel $K$ and Pythagorean theorem by orthogonality of summands; once again thanks to orthogonality of summands, the image under $\alpha$ of the set $X_1\star X_2$ is affinely independent and our conclusion follows from \autoref{lemma:keylem}\emph{(b)}.
\end{proof}
\begin{rmk}
 Let us note that this lemma could be proved completely algebraically, without restriction to the case Schoenberg kernels. The proof relies on a similar type of argument, but needs some technical computations. Details, in the case of \CND\ kernels, can be found in \cite{Bo2}.
\end{rmk}
\subsection{An obstruction}\label{sec:2.5}
Let $G = (V, E)$ be a connected graph, the path metric $\partial$ on $V$ is defined by setting $\partial(x,y)$ to be the length of shortest path connecting $x, y\in V$. 
\begin{rmk}
If $G_0\subset G$ is a subgraph such that the inclusion is an isometric embedding of metric spaces, then the metric on $G_0$ inherits spectral properties from $G$. In particular, if $G_0$ is a subgraph such that its metric is not \CSND, the metric of $G$ does not have this property. 
\end{rmk}
\begin{lemma}\label{lemma:obst}
If the shortest cycle in a graph has even length, then the metric is not \CSND.
\end{lemma}
\begin{proof}
Since the shortest cycle is always isometrically embedded, it is enough to restrict our attention to this cycle of even length. On such a cycle one can find a set of four vertices $(x_i)_{i=1}^4$ such that two of them are antipodes of the other two. In particular, the matrix of distances between them is of the form 
$$D=[\partial(x_i,x_j)]_{i,j=1}^4=\left[\begin{array}{cccc}0&k&n+k&n\\k&0&n&n+k\\n+k&n&0&k\\n&n+k&k&0\end{array}\right]$$
It is now easy to see that the vector $v = (1,-1,1,-1)^\top$ is in kernel of the matrix $D$, so the assertion follows from \autoref{lemma:keylem}\emph{(a)}.
\end{proof}
\section{Applications and graph theoretic examples}\label{sec:3}
The comb product of graphs (cf.~\cite{HO, Ob}) is constructed in a similar way to Markov product of graphs (cf. \autoref{sec:2.4}): given two graphs $(V_1,E_1), (V_2,E_2)$ and distinguished point $v_2\in V_2$, we form a graph $V= V_1\triangleright_{v_2} V_2=V_1\triangleright V_2$ by attaching to each point of $V_1$ a copy of $V_2$ based in~$v_2$. Formally, enumerate elements of $V_1=\{v_1, v_2, \ldots\}$, define $V^{(0)}=V_1$ and $V^{(n)}=V_1^{(n-1)}{}_{\phantom{1}{v_n}}\!\star_{v_2} V_2$ with an obvious embedding $V^{(n-1)}$ into $V^{(n)}$. Set $V=\varinjlim V^{(n)}$. Thanks to the construction lemma from \autoref{sec:2.4} and desription of $V$ as inductive limit, we have the following
\begin{proposition}
If $V_i$ are graphs whose metrics are \CSND, then so is the metric on $V_1\triangleright V_2$.
\end{proposition}
The free product of graphs is a little more laborious to construct and more details can be found in~\cite{ALS}. We begin with defining free product of pointed sets: for a family $(S_i, e_i)_{i\in I}$ their free product is a pointed set $(S,e)=(\fp\limits_{i\in I}S_i,e)$ defined as follows: $$S=\{e\}\cup\{s_1\ldots s_m: s_k\in S_{i_k}\backslash\{e_{i_k}\}\textrm{ and }i_1\neq i_2\neq\ldots\neq i_m, m\in\mathbb{N}\}.$$ The set of vertices of free product of graphs is $\fp\limits_{i\in I}(V_i,v_i)$. The set of edges of the free product of graphs $(V_i,E_i)_{i \in I}$ is defined as follows: $$\fp\limits_{i\in I}E_i=\left\{\{vu,vu'\}:\ \{v,v'\}\in\bigcup_{i\in I}E_i\textrm{ and }u,vu,v'u\in\fp\limits_{i\in I}V_i\right\}.$$ Because in $G_1\hexstar G_2$ one can find an increasing sequence of graphs $\{H_n\}$ satisfying:
\begin{itemize}
\item $H_n$ is a result of finitely many operations of Markov sums with summands $G_1, G_2$ and previously constructed graphs,
\item $H_n$ is isometrically embedded in $G_1\hexstar G_2$,
\item $G_1\hexstar G_2$ is an union of $H_n$'s. 
\end{itemize}
we have the following
\begin{proposition}
Given two graphs with distinguished vertices $(V_i, v_i)$, $i=1,2$, whose metrics are \CSND, the metric on $V_1\hexstar V_2$ is also \CSND.
\end{proposition}
The most important application of this lemma is when we consider the Cayley graphs of discrete groups. We have
\begin{proposition}
If $\Gamma_i=\ls S_i|R_i \rs$, $i=1,2$ are discrete groups such that the path metric on $Cay(\Gamma_i, S_i)$ is \CSND, so is the metric on $Cay(G_1\hexstar G_2, S_1\cup S_2)$.
\end{proposition}
\begin{proof}
It is enough to observe $Cay(\Gamma_1,S_1)\hexstar Cay(\Gamma_2,S_2)=Cay(G_1\hexstar G_2, S_1\cup S_2)$, where the first $\hexstar$ denotes free product of graphs, and the other denotes the free product of groups.
\end{proof}
Observe that graphs having $\leq3$ vertices have \CSND\ metrics because of arithmetic reasons, so this gives us the following 
\begin{corollary}\label{corollary:regtree}
 The path metric on the graph $Cay(\fp\limits_{i\in I}\bigslant{\mathbb{Z}}{2\mathbb{Z}},\{1_i:i\in I\})$ (the Cayley graph of free Coxeter group or the $\# I$-regular tree), is \CSND. 
\end{corollary}
\begin{corollary}
 The path metric on any tree is \CSND. In particular, so is the word metric on free groups.
\end{corollary}
\begin{proof}
 It is enough to observe that any tree $T$ can be isometrically embedded in a regular tree: let $\kappa=\sum\limits_{t\in T} \mathrm{deg}(t)$ and for $\kappa$-regular tree the assertion holds. \end{proof}
\begin{proposition}\label{prop:complete}
Metric of every complete graph $K_{n+1}$ is \CSND\ for any $n\in\mathbb{N}$.
\end{proposition}
\begin{proof}
Since every subgraph of complete graph is again complete, it is enough to consider $\lambda$~with~full support. Let $D$ be the distance matrix for $K_{n+1}$. One has
$$D=\left[ \begin{array}{ccccc}
0 & 1 & 1 & \ldots & 1\\
1 & 0 & 1 & \ldots & 1\\
1 & 1 & 0 & \ldots & 1\\
\vdots & \vdots & \vdots & \ddots & \vdots\\
1 & 1 & 1 & \ldots & 0\\
\end{array} \right].
$$
Then for $\lambda\!=\!~(\!\lambda_1\!,\lambda_2\!,\ldots\!,\lambda_n\!,\lambda_{n+1})^\top$,
setting $\Lambda=\sum\limits_{k=1}^{n+1}\lambda_k$, we have: $D\lambda=(\Lambda-\lambda_1,\Lambda-\lambda_2, \ldots,\Lambda-\lambda_n,\Lambda-\lambda_{n+1})^\top\neq0$ 
for $\lambda\neq0$, which finishes the proof by \autoref{lemma:keylem}\emph{(a)}.
\end{proof}
\begin{corollary}
\CSND\ is a property of a metric, not a group.
\end{corollary}
\begin{proof}
Let us consider the group $\bigslant{\mathbb{Z}}{4\mathbb{Z}}$. The metric on $Cay(\bigslant{\mathbb{Z}}{4\mathbb{Z}},\{0,1,2,3\})$ is \CSND\ by \autoref{prop:complete}, while by~\autoref{lemma:obst}, the metric on $Cay(\bigslant{\mathbb{Z}}{4\mathbb{Z}},\{1,3\})$ is not.
\end{proof}
\begin{proposition}
 $(2n+1)$-gon has \CSND.
\end{proposition}
\begin{proof}
Let us identify the $(2n+1)$-gon with the set $X=\{1,2,\ldots,2n+1\}$. For a subset $A\subseteq X$, let $\tau_A=\chi_A-\chi_{A'}$ be the function which is $+1$ on set $A$ and $-1$ on its complement. Up to a normalization constant, one has a quadratic embedding: $k\mapsto \tau_{\{1,\ldots,2k\}}$, for $1\leq k\leq n$; $(n+k)\mapsto \tau_{\{1,\ldots,2k-1\}'}$, for $1\leq k\leq n+1$ and it is now easy to see that the image of $(2n+1)$-gon under this mapping is linearly independent.
\end{proof}
Yet another evidence that the property \CSND\ has much more to do with the graph structure, not the group structure, is contained in the following 
\begin{proposition}
The property that the path metric is \CSND\ is not stable under amalgamated free product.
\end{proposition}
\begin{proof}
Consider $G=\left(\bigslant{\mathbb{Z}}{9\mathbb{Z}}\right)\fp\limits_{\left(\bigslant{\mathbb{Z}}{3\mathbb{Z}}\right)}\left(\bigslant{\mathbb{Z}}{9\mathbb{Z}}\right)=\ls a,b|a^9,a^3b^{-3}\rs$, with the natural generating set coming from \CSND\ graphs $\bigslant{\mathbb{Z}}{9\mathbb{Z}}=\ls a|a^9\rs$. The shortest relation in the group $G$ ($a^3=b^3$) is of even length and the conclusion follows from the \autoref{lemma:obst}.
\end{proof}

\begin{df}{\label{def:brac}}
 Let $S$ be a semigroup, for a positive integer $k\in\mathbb{N}$ define the alternating product of length $k$, $\ls\cdot,\cdot\rs^k\colon S\times S\to S$ by the formula $\ls s,t\rs^k=stst\ldots st$ (each of the letter $s, t$ appears exactly $m$ times for $k=2m$) and $\ls s,t\rs^k=stst\ldots sts$ (the letter $t$ appears exactly $m$ times, while $s$ appears $m+1$ times, when $k=2m+1$).
\end{df}
\begin{df}{\label{def:artcox}}
 Let $S$ be a set and let a list of coefficients $m_{s,t}=m_{t,s}\in\{2,3,\ldots\}\cup\{\infty\}$, where $s\neq t\in S$, be given. A group $G$ given by the presentation $G=\ls S|R\rs$ is called an Artin group if the set of relations is of the following form: $R=\{\ls s,t\rs^{m_{s,t}}=\ls t,s\rs^{m_{s,t}}:s,t\in S\}$ (by convention, $\ls s,t\rs^\infty=\ls t,s\rs^\infty$ is read as: none of $\ls s,t\rs^k=\ls t,s\rs^k$ with $k=2,3,\ldots$ holds). A Coxeter group with generating set $S$ and list of coefficients $m_{s,t}=m_{t,s}\in\{2,3,\ldots\}\cup\{\infty\}$, $s\neq t\in S$, is a quotient of an Artin group with generating set $S$ and list of coefficients $m_{s,t}$ by the normal subgroup generated by $\{s^2:s\in S\}$.
\end{df}

\begin{proof}[Proof of \autoref{theorem:2}]
Assume that relation of the form $\ls s,t\rs^{m_{s,t}}=\ls t,s\rs^{m_{s,t}}$ holds in $G$ for some $s, t\in S$, $s\neq t$ and $2\leq m_{st}<\infty$. Then its word metric is not \CSND. Indeed, for $s\neq t\in S$ such that $m_{s,t}\geq2$ is minimal among all the coefficients, the subgroup generated by those elements, $\ls\{s,t\}\rs\subseteq G$, provides a cycle of length $2m_{s,t}\geq4$ in the Cayley graph (with respect to the generating set $S$), which is isometrically embedded (due to minimality). Thus one direction follows by \autoref{lemma:obst}, while the other implication is already shown in~\autoref{corollary:regtree}.
\end{proof}
\begin{rmk}
 The proof of the fact that the word length of a free Coxeter group is \CSND\ can be also given in terms of its Tits representation, describing appropriate combinatorial structures (root system and Weyl chambers).
\end{rmk}

\begin{proof}[Proof of \autoref{theorem:1}]
Let $\alpha$ be the quadratic embedding for the kernel $K$, we may assume that the target Hilbert space $\mathcal{H}$ of $\alpha$ is $\#X$-dimensional. Since $\{\alpha(x):x\in X\}$ is a finite set of points lying on a single sphere (by \autoref{lemma:affind}\emph{(d)} and \autoref{lemma:keylem}\emph{(b)}) with center $e\in\mathcal{H}$ and radius $r>0$, by composing $\alpha$ with translation by $-e$ we can assume that they all have same norm $\|\alpha(x)\|=r$.

Now if $0\notin \textrm{aff span}\{\alpha(x):x\in X\}$ we are done thanks to \autoref{lemma:fin}: 
\[K(x,y)=\|\alpha(x)-\alpha(y)\|^2 = \|\alpha(x)\|^2-2\ls\alpha(x)|\alpha(y)\rs+\|\alpha(y)\|^2=2r^2-2\ls\alpha(x)|\alpha(y)\rs\] 
and $\ls\alpha(x)|\alpha(y)\rs$ is strictly positive definite by \autoref{ex:1}.

If $0\in \textrm{aff span}\{\alpha(x):x\in X\}$, take any non-zero $v\in\{\alpha(x):x\in X\}^{\bot}$ (which is possible due to \autoref{lemma:fin} and assumption that $\dim\mathcal{H}=\# X$) and once again compose $\alpha$ with translation by $v$. Now, as $\|\alpha(x)+v\|^2=r^2+\|v\|^2$ does not depend on $x$, the vectors $\{\alpha(x)+v:x\in X\}$ remain affinely independent. In particular, $0\notin\textrm{aff span}\{\alpha(x)+v:x\in X\}$ (if $\sum\limits_{x\in X} \lambda(x)(\alpha(x)+v)=0$, then, as $v\perp\alpha(x)$ for all $x\in X$, $\sum\limits_{x\in X} \lambda(x)=0$ and consequently $\sum\limits_{x\in X} \lambda(x)\alpha(x)=0$, so by \autoref{lemma:affind}\emph{(c)}, $\lambda(x)=0$ for all $x\in X$) and now $\alpha(x)+v$ can play the role of $\alpha(x)$ in the calculations from the previous step.
\end{proof}

\section*{Appendix}
We provide a simple proof of a theorem due to Schoenberg; it will be useful in producing a continuous example of a \CSND\ kernel and in \autoref{cor:app} which is a continuous analogue of \autoref{theorem:2}. Recall the following classical result:
\begin{thm}\label{theorem:FTP}
 $$\displaystyle\int_{\mathbb{R}^k}e^{-2t\pi\|x\|}e^{-2\pi i\ls\xi|x\rs}\dd x=\displaystyle\frac{t\cdot c_k}{(t^2+\|\xi\|^2)^{\frac{k+1}{2}}},$$
for $t>0$, where the constant $c_k=\Gamma(\frac{k+1}{2})\sqrt{\pi}^{(-k-1)}$ depends only on the dimension.
\end{thm}
For the proof we refer to~\cite[chapter~1]{SW}.
\begin{thm}[\cite{Sc2}]\label{theorem:Schoenberg2}
The euclidean metric on $\mathbb{R}^k$ is \CSND.
\end{thm}
\begin{proof}
Observe first, that $\lim\limits_{n\to\infty}n(e^{-\frac{1}{n}\|x\|}-1)=-\|x\|$. Fix a finite set of points  $X\subseteq\mathbb{R}^k$ and a non-zero sequence of coefficients $\lambda\colon X\to\mathbb{C}$ such that $\sum\limits_{x\in X}\lambda(x)=0$. We have:
\begin{dmath*}
-\sum\limits_{x,y\in X}\!\!\!\lambda(x)\overline{\lambda(y)}\|x-y\|=
\lim_{n\to\infty}n\sum\limits_{x,y\in X}\lambda(x)\overline{\lambda(y)}(e^{-\frac{1}{n}\|x-y\|}-1)=
\lim_{n\to\infty}n\sum\limits_{x,y\in X}\lambda(x)\overline{\lambda(y)}e^{-\frac{1}{n}\|x-y\|} 
\end{dmath*}
where we use that $\sum\limits_{x\in X}\lambda(x)=0$ in the last equality. Write the inverse Fourier transform (with $t=\frac{1}{2n\pi}$) as in the previous theorem. We continue with:
\begin{dmath*}
 =\lim_{t\to0}n\sum\limits_{x,y\in X}\lambda(x)\overline{\lambda(y)}\int_{\mathbb{R}^k}\frac{t\cdot c_k}{(t^2+\|\xi\|^2)^{\frac{k+1}{2}}}e^{2\pi i\ls\xi|x-y\rs}\dd \xi
 \end{dmath*} \begin{dmath*}
 =\frac{c_k}{2\pi}\lim_{t\to0}\int_{\mathbb{R}^k}\frac{\sum\limits_{x,y\in X}\lambda(x)\overline{\lambda(y)}}{(t^2+\|\xi\|^2)^{\frac{k+1}{2}}}e^{2\pi i\ls\xi|x\rs}\overline{e^{2\pi i\ls\xi|y\rs}}\dd \xi
 =\frac{c_k}{2\pi}\lim_{t\to0}\int_{\mathbb{R}^k}\bigg|\sum\limits_{x\in X}\lambda(x)e^{2\pi i\ls\xi|x\rs}\bigg|^2\frac{\dd \xi}{(t^2+\|\xi\|^2)^{\frac{k+1}{2}}}
 =\frac{c_k}{2\pi}\int_{\mathbb{R}^k}\bigg|\sum\limits_{x\in X}\lambda(x)e^{2\pi i\ls\xi|x\rs}\bigg|^2\frac{\dd \xi}{\|\xi\|^{k+1}}
 \end{dmath*}
the last equality is due to Lebesgue monotone convergence theorem. Observe that the expression under integral is non-zero, since the functions $e^{2\pi i\ls\cdot|x\rs}$ are linearly independent (in one dimensional case it is a well known fact thanks to the Vandermonde determinant; the general case may be reduced to the one-dimensional case: because $X\subseteq\mathbb{R}^k$ is finite, one may find a line $\mathbb{R}\mathbf{v}$ intersecting \[\bigcup\limits_{\substack{x,y\in X,\\ x\neq y}}\{x-y\}^{\perp}\neq\mathbb{R}^k\] only in $\mathbf{0}\in\mathbb{R}^k$), so the integral is strictly positive.
\end{proof}
\begin{rmk}
 This type of reasoning is classical in harmonic analysis and a similar result could be obtained a different method, cf. \cite{Bal, Bax, Mic, MS, Sc1, Sc2}. We owe the main idea of the proof to M. Bożejko.
\end{rmk}
A metric space $(X,d)$ is called an $\mathbb{R}$-tree, if any two points $x,y\in X$ can be connected by an unique arc (topological embedding of a closed interval) and this arc is a geodesic segment (i.e. the embedding of a closed interval is isometric). A reasonable reference for the theory and applications of $\mathbb{R}$-trees is \cite{MB}. A notion of $\mathbb{R}$-graph is not well established in the literature, altough the general idea is rather folklore: it should locally look like an $\mathbb{R}$-tree, but there might be some $S^1$ embedded into it. There are several ways to formalize of this idea, but for our purposes it will be enough to consider the case where we do not allow circles of arbitrary small length in a neighbourhood of a single point. We call a metric space $(X,d)$ an $\mathbb{R}$-graph if for any point $x\in X$ there exists an $\varepsilon>0$ such that the ball $B(x,\varepsilon)$ with the restricted metric is an $\mathbb{R}$-tree. 
\begin{appcor}\label{cor:app}
A metric on $\mathbb{R}$-graph $X$ is conditionally strictly negative definite if and only if $X$ is an $\mathbb{R}$-tree.
\end{appcor}
\begin{proof}
We begin with the observation, that an $\mathbb{R}$-graph $X$ is an $\mathbb{R}$-tree if and only if there is no isometrically embedded $S^1$. Indeed, if there is one, then there are several arcs joining the antipodal points of the circle. If there are multiple arcs joining some points $x,y$ and they are geodesic segments, they constitute an isometrically embedded circle. Now, if there is an isometrically embedded circle, an analogue of \autoref{lemma:obst} (for $k,n\in\mathbb{R}_{>0}$) gives us implication in one direction. To prove the converse it is enough to observe that for given finite set of points in $X$, they lie on finitely many branches of the $\mathbb{R}$-tree, so we conclude by recalling the wedge sum \autoref{lemma:ms} and Schoenberg's \autoref{theorem:Schoenberg2}.
\end{proof}
The above corollary recovers, in particular, a well-known result of M. Bożejko (\cite{Bo1}) that the metric on $\mathbb{R}$-tree is conditionally negative definite. Of course, a similar way one can produce a list of less obvious examples, like $\fp\limits_{i\in I}\mathbb{R}^{k_i}$. Although they are all $\mathsf{CAT}(0)$ spaces, so they do not give new examples of spaces on which a group could act to have the Haagerup approximation property. On the other hand, a plane with jungle river metric or SNCF metric are examples that might not come to mind quickly.

\end{document}